\documentclass[A4]{article}
\usepackage{amscd}
\usepackage{amssymb}
\usepackage{amsmath}
\usepackage{amsthm}
\usepackage{latexsym}
\usepackage{upref}
\usepackage{epsfig}
\usepackage{mathrsfs}
\usepackage{float}
\usepackage{indentfirst}
\usepackage{hyperref}
\hypersetup{
    colorlinks=true,
    linkcolor=blue,
    citecolor=blue,
    filecolor=blue,
    urlcolor=blue,
}

\usepackage{graphics,graphicx,color}
\usepackage{floatflt}
\usepackage{cite}
\usepackage{enumitem}
\setlist{font=\normalfont} 
\usepackage{fancyhdr} 
\usepackage[T1]{fontenc}
\usepackage{mathptmx} 

\usepackage{pgf,tikz}
\usepackage{mathrsfs}
\usetikzlibrary{arrows}
\usepackage{pstricks-add}

\theoremstyle{plain}
\newtheorem{theorem}{Theorem}[section]
\newtheorem{lemma}[theorem]{Lemma}

\newtheorem{corollary}[theorem]{Corollary}
\theoremstyle{definition}

\newtheorem{example}[theorem]{Example}
\theoremstyle{remark}

\numberwithin{equation}{section}

\makeatletter
\def\th@plain{%
  \thm@notefont{}
  \itshape 
}
\def\th@definition{%
  \thm@notefont{}
  \normalfont 
} \makeatother

\setlist{font=\normalfont}

\newcommand{\N}{\mathbb{N}}
\newcommand{\set}[1]{\{#1\}}

\newcommand{\cset}[2]{\set{{#1}\colon{#2}}}

\newcommand{\gen}[1]{\langle#1\rangle}
\newcommand{\abs}[1]{|#1|}

\newcommand{\dCay}[1]{\overrightarrow{\mathrm{Cay}}\,{(#1)}}
\newcommand{\dcCay}[1]{\overrightarrow{\mathrm{Cay}}{_c\,}{(#1)}}
\newcommand{\Cay}[1]{\mathrm{Cay}\,{(#1)}}

\DeclareMathOperator{\p}{\mathfrak{p}}

\newcommand{\indeg}[1]{\mathrm{indeg}\,{#1}}
\newcommand{\outdeg}[1]{\mathrm{outdeg}\,{#1}}

\begin{document}
\title{\textbf{An algorithm for finding minimal generating sets of finite groups}\footnote{Part of this work was presented at the conference Semigroups and Groups, Automata, Logics \mbox{(SandGAL 2019)}, Cremona, Italy, June 10-13, 2019.}}
\author{Tanakorn Udomworarat \quad and\quad Teerapong Suksumran\footnote{Corresponding author.}\,\,\,\href{https://orcid.org/0000-0002-1239-5586}{\includegraphics[scale=1]{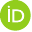}}\\[5pt]
Department of Mathematics\\
Faculty of Science, Chiang Mai University\\
Chiang Mai 50200, Thailand\\[5pt]
\texttt{tanakorn\_u@cmu.ac.th} (T. Udomworarat)\\
\texttt{teerapong.suksumran@cmu.ac.th} (T. Suksumran)}
\date{}
\maketitle

\begin{abstract}
In this article, we study connections between components of the Cayley graph $\mathrm{Cay}(G,A)$, where $A$ is an arbitrary subset of a group $G$, and cosets of the subgroup of $G$ generated by $A$. In particular, we show how to construct generating sets of $G$ if $\mathrm{Cay}(G,A)$ has finitely many components. Furthermore, we provide an algorithm for finding minimal generating sets of finite groups using their Cayley graphs. 
\end{abstract}

\section{Introduction}\label{sec: introdcution}
The problem of determining (minimal) generating sets of groups has been studied widely; see, for instance, \cite{MR3814347, MR2361465, MR572868, MR3073659, MR2023255, MR1289999, MR643291}. Although the existence of a generating set of a certain group is known, it might be a difficult task to explore it. In fact, there is no general effective method to determine a generating set of a given group. From a complexity point of view, MIN-GEN---finding a minimum size generating set---for groups is in DSPACE\,($\log^2{n}$); see Proposition 3 of \cite{MR2296106}.

Cayley graphs prove useful in studying various mathematical structures \cite{MR1957965, MR1927074, MR1939667, MR2228536, MR1955389, MR2574829, MR1934292, MR2548555} and have many applications in several fields \cite{MR2548552, MR2195356, MR2080457, MR2455531, MR2885431}. It is clear that a Cayley graph of any group $G$ encodes a large amount of information about algebraic, combinatorial, and geometric structures of $G$. In addition, Cayley graphs allow us to visualize groups and then to examine properties of groups in a convenient way. Often, a Cayley graph of a (finite) group $G$ relative to $A$ is drawn with the condition that $A$ be a generating set of $G$ (and even if $A = A^{-1}$; e.g. in geometric group theory). In this case, the Cayley graph is strongly connected. In fact, it is a standard result in the literature that a Cayley graph $\Cay{G, A}$ is connected if and only if $A$ generates $G$. In the present article, we investigate Cayley graphs in a more general setting by relaxing the requirement that $A$ be a generating set of $G$. In particular, we count the number of components of a Cayley graph in terms of the index of a subgroup. As a consequence of this result, we are able to construct a generating set of a certain group whenever its Cayley graph has finitely many components. This leads to an algorithm for finding some minimal generating sets of finite groups. We remark that some results presented in the next section are known in algebraic graph theory; see, for instance, \cite[Lemma 5.1]{MR1957965}, \cite[p.1]{MR2320601}, and \cite[pp. 302--303]{MR1927074}.

\section{Main results}
\subsection{Definitions and basic properties}\label{subsec: basic property}

Let $G=(V,E)$ be a graph. The equivalence relation $\p$ on $V$ is defined by $u\p v$ if and only if either $u=v$ or there is a path from $u$ to $v$ in $G$. We say that $C$ is a component of $G$ if and only if there is an equivalence class $X$ of the relation $\p$ such that $C$ is the subgraph of $G$ induced by $X$, that is, $C=G[X]$. Note that $u\p v$ if and only if $u$ and $v$ are in the same component of $G$.

Let $G$ be a (finite or infinite) group and let $A$ be an arbitrary subset of $G$. Recall that the {\it Cayley digraph} of $G$ with respect to $A$, denoted by $\dCay{G, A}$, is a directed graph with the vertex set $G$ and the arc set $\cset{(g, ga)}{g\in G, a\in A\textrm{ and }a\ne e}$. Remark that we exclude the identity $e$ in order to avoid loops in Cayley digraphs.  The (undirected) {\it Cayley graph} of $G$ with respect to $A$, denoted by $\Cay{G, A}$, is defined as a graph with the vertex set $G$ such that $\set{u, v}$ is an edge if and only if $u = va$ or $v = ua$ for some $a\in A\setminus\set{e}$. It is not difficult to check that $\Cay{G, A}$ is the underlying graph of $\dCay{G, A}$; that is, the vertex sets of $\Cay{G, A}$ and $\dCay{G, A}$ coincide and $\set{u, v}$ is an edge in $\Cay{G, A}$ if and only if $(u, v)$ or $(v, u)$ is an arc in $\dCay{G,A}$. To nonidentity elements of $A$, we can associate distinct colors, labeled by their names. The {\it Cayley color digraph} of $G$ with respect to $A$, denoted by $\dcCay{G, A}$, consists of the digraph $\dCay{G, A}$ in which any arc $(g, ga)$ is given color $a$ for all  $a\in A\setminus\set{e}$.

Next, we mention some basic properties of Cayley graphs and Cayley digraphs.

\begin{theorem}\label{thm: indegree & outdegree of dCay}
Let $G$ be a group and let $A$ be a finite subset of $G$. Then  
$$
\indeg{v}=\outdeg{v}=
\begin{cases}
\abs{A} & \textrm{ if } e\notin A;\\
\abs{A}-1 & \textrm{ if } e\in A 
\end{cases}
$$
for all vertices $v$ in $\dCay{G,A}$. Therefore, $\dCay{G,A}$ is regular.
\end{theorem}
\begin{proof}
Note that $va=vb$ if and only if $a=b$ for all $u,v\in G, a,b\in A$. In the case when $e\notin A$, we obtain
$$\outdeg{v}=\abs{\set{va:a\in A}}=\abs{A}$$ and $$\indeg{v}=\abs{\set{va^{-1}:a\in A}}=\abs{\set{a^{-1}:a\in A}}=\abs{A}.$$
In the case when $e\in A$, we obtain
$$\outdeg{v}=\abs{\set{va:a\in A-\set{e}}}=\abs{A-\set{e}}=\abs{A}-1$$ and
\begin{equation*}\tag*{\qedsymbol}
\indeg{v}=\abs{\set{va^{-1}:a\in A-\set{e}}}=\abs{\set{a^{-1}:a\in A-\set{e}}}=\abs{A-\set{e}}=\abs{A}-1.
\end{equation*}
\let\qed\relax
\end{proof}

\begin{theorem}\label{thm: degree of Cayley graph}
Let $G$ be a group and let $A$ be a finite subset of $G$ not containing $e$. Then  
\begin{equation}
\deg{v}=2\abs{A}-\abs{A\cap A^{-1}},
\end{equation}
where $A^{-1} = \cset{a^{-1}}{a\in A}$, for all vertices $v$ in $\Cay{G,A}$. Therefore, $\Cay{G,A}$ is regular. 
\end{theorem}
\begin{proof}
Suppose that $vA=\set{va:a\in A}$ and $vA^{-1}=\set{va^{-1}:a\in A}$. Then $vA\cap vA^{-1}=v(A\cap A^{-1})$. Note that
\begin{align*}
\textrm{$\set{u,v}\in E(\Cay{G,A})$}\quad &\Leftrightarrow\quad \textrm{$(u,v)\in E(\dCay{G,A})$ or $(v,u)\in E(\dCay{G,A})$}\\
{} &\Leftrightarrow\quad \textrm{$u=va^{-1}$ for some $a\in A$ or $u=va$ for some $a\in A$}.
\end{align*}
Hence, 
\begin{align*}
\deg v &=\abs{\set{u\in G:\set{u,v}\in E(\Cay{G,A})}}\\
{} &= \abs{vA\cup vA^{-1}}\\
{} &= \abs{vA}+\abs{vA^{-1}}-\abs{vA\cap vA^{-1}}\\
{} &= \abs{vA}+\abs{vA^{-1}}-\abs{v(A\cap A^{-1})}\\
{} &= 2\abs{A}-\abs{A\cap A^{-1}}\tag*{\qedsymbol}
\end{align*}
\let\qed\relax
\end{proof}

The following lemma gives a characterization of the existence of paths in Cayley graphs, which is quite well known in the literature.
\begin{lemma}[See, e.g., Lemma 5.1 of \cite{MR1957965}]\label{lem: path from g to h}
Let $g$ and $h$ be distinct elements in a group $G$ and let $A\subseteq G$. Then there is a path from $g$ to $h$ in $\Cay{G,A}$ if and only if  $g^{-1}h = a_1^{\varepsilon_1}a_2^{\varepsilon_2}\cdots a_n^{\varepsilon_n}$ for some $a_1, a_2,\ldots, a_n\in A$, $\varepsilon_1, \varepsilon_2,\ldots, \varepsilon_n\in\set{\pm 1}$.
\end{lemma}

\subsection{Components of Cayley graphs and Cayley digraphs}

In what follows, if $A$ is a subset of a group $G$, then $\gen{A}$ denotes the subgroup of $G$ generated by $A$. That is, $\gen{A}$ is the smallest subgroup of $G$ containing $A$. Henceforward, $A$ is  a subset of a (finite or infinite) group $G$ unless stated otherwise. 

An equivalence class of the relation $\p$ induced by the Cayley graph $\Cay{G, A}$, defined in the beginning of Section \ref{subsec: basic property}, turns out to be a coset of $\gen{A}$ in $G$, as shown in the following theorem.

\begin{theorem}\label{theorem: coset of gen{A} iff p relation}
Let $u, v\in G$. Then $u$ and $v$ are in the same coset of $\gen{A}$ in $G$ if and only if $u\p v$, where $\p$ is the equivalence relation induced by $\Cay{G,A}$.
\end{theorem}

\begin{proof}
$(\Rightarrow)$ Let $X$ be a coset of $\gen{A}$ in $G$. Then $X=g\gen{A}$ for some $g\in G$. Let $u,v\in X$. Then $u=ga_1^{\varepsilon_1}a_2^{\varepsilon_2}\cdots a_n^{\varepsilon_n}$ and $v=gb_1^{\delta_1}b_2^{\delta_2}\cdots b_m^{\delta_m}$, where $a_1,a_2,\ldots,a_n,b_1,b_2,\ldots,b_m\in A$ and $\varepsilon_1,\varepsilon_2,\ldots,\varepsilon_n,\delta_1,\delta_2,\ldots,\delta_m\in\set{\pm 1}$. In the case when $u\neq v$, $$u^{-1}v=a_n^{-\varepsilon_n}\cdots a_2^{-\varepsilon_2}a_1^{-\varepsilon_1}b_1^{\delta_1}b_2^{\delta_2}\cdots b_m^{\delta_m}$$ implies that there is a path from $u$ to $v$ by Lemma \ref{lem: path from g to h}. Thus $u\p v$.

$(\Leftarrow)$ Let $u,v\in G$ and suppose that $u\p v$. If there exists a path from $u$ to $v$ in $\Cay{G,A}$, then $u^{-1}v=a_1^{\varepsilon_1}a_2^{\varepsilon_2}\cdots a_n^{\varepsilon_n}$, where $a_1,a_2,\ldots,a_n\in A, \varepsilon_1,\varepsilon_2,\ldots,\varepsilon_n\in\set{\pm 1}$. Suppose that $u\in g\gen{A}$ for some $g\in G$. Then $u=gb_1^{\delta_1}b_2^{\delta_2}\cdots b_m^{\delta_m}$ with \mbox{$b_1,b_2,\ldots,b_m\in A$}, $\delta_1,\delta_2,\ldots,\delta_m \in\set{\pm 1}$. Thus $v=gb_1^{\delta_1}b_2^{\delta_2}\cdots b_m^{\delta_m}a_1^{\varepsilon_1}a_2^{\varepsilon_2}\cdots a_n^{\varepsilon_n}\in g\gen{A}$. 
\end{proof}

\begin{corollary}\label{cor: equivalent class iff coset of gen{A}}
Let $X\subseteq G$. Then $X$ is an equivalence class of the relation $\p$ induced by $\Cay{G,A}$ if and only if $X$ is a coset of $\gen{A}$ in $G$.
\end{corollary}

\begin{corollary}\label{cor: component of Cay is in the form of coset}Let $G$ be a group and let $A\subseteq G$.

\begin{enumerate}
    \item\label{item: component of Cayley graph Cay(G,A)} $C$ is a component of $\Cay{G,A}$ if and only if there is a unique coset $X$ of $\gen{A}$ such that $C = \Cay{G,A}[X]$.
    \item$C$ is a component of $\dCay{G,A}$ if and only if there is a unique coset $X$ of $\gen{A}$ such that $C = \dCay{G,A}[X]$.
\end{enumerate}
\end{corollary}

In general, finding the subgroup of $G$ generated by $A$ might be a complicated and \mbox{tedious} task. Corollary \ref{cor: component of Cay is in the form of coset} enables us to find this subgroup by looking at the component of $\Cay{G, A}$ that contains the identity of $G$ (see Theorem \ref{thm: component in terms of cosets}). Furthermore, it \mbox{indicates} that another component of $\Cay{G, A}$ is simply a left translation of the \mbox{identity} component (see the proof of Theorem \ref{thm: two components isomorphic}). We remark that part of Corollary \ref{cor: component of Cay is in the form of coset} \eqref{item: component of Cayley graph Cay(G,A)} is known in the literature; see, for instance, \cite[p. 1]{MR2320601}.

\begin{theorem}\label{thm: component in terms of cosets}
The subgroup $\gen{A}$ is equal to the set of vertices in the component of $\Cay{G,A}$ containing the identity of $G$. In general, $C$ is a component of $\Cay{G,A}$ containing a vertex $v$ if and only if the vertex set of $C$ equals $v\gen{A}$.
\end{theorem}

\begin{proof}
Let $C$ be the component of $\Cay{G,A}$ containing the identity of $G$. By Corollary \ref{cor: component of Cay is in the form of coset}, $C=\Cay{G,A}[g\gen{A}]$ for some $g\in G$. Since $C$ contains the identity of $G$, that is, $e\in g\gen{A}$, it follows that $g\in\gen{A}$ and so $g\gen{A}=\gen{A}$. The remaining statement can be proved in a similar fashion.
\end{proof}

\begin{theorem}\label{thm: two components isomorphic}
If $B$ and $C$ are components of $\Cay{G,A}$, then $B$ and $C$ are isomorphic as graphs.
\end{theorem}

\begin{proof}
Suppose that $B$ and $C$ are components of $\Cay{G,A}$. From Corollary \ref{cor: component of Cay is in the form of coset}, we obtain $B=\Cay{G,A}[g_1\gen{A}]$ and $C=\Cay{G,A}[g_2\gen{A}]$ for some $g_1,g_2\in G$. Let $\varphi$ be a map defined by $\varphi(x)=g_2g_1^{-1}x$ for all $x\in g_1\gen{A}$. It is straightforward to check that $\varphi$ is a graph isomorphism from $B$ to $C$. So $B$ and $C$ are isomorphic.
\end{proof}

Another application of Corollary \ref{cor: component of Cay is in the form of coset} reveals a geometric aspect of Cayley digraphs and Cayley graphs: they are disjoint unions of smaller Cayley digraphs (or graphs).

\begin{theorem}
Let $G$ be a group and let $A\subseteq G$. If $C_i, i\in I$, are all the components of $\dCay{G, A}$ and if $v_i$ is a vertex in $C_i$ for all $i\in I$, then
$$\dCay{G,A}=\dot{\bigcup\limits_{i\in I}}\dCay{G,A}[v_i\gen{A}],$$
where the dot notation indicates that the union is disjoint.
\end{theorem}

\begin{corollary}
Let $G$ be a group and let $A\subseteq G$. If $C_i, i\in I$, are all the components of $\Cay{G, A}$ and if $v_i$ is a vertex in $C_i$ for all $i\in I$, then
$$\Cay{G,A}=\dot{\bigcup\limits_{i\in I}}\Cay{G,A}[v_i\gen{A}].$$
\end{corollary}

Next, we show that the number of components of $\Cay{G, A}$ is indeed the number of cosets of $\gen{A}$ in $G$. This result refines the well known fact that the Cayley graph $\Cay{G, A}$ is connected if and only if $A$ generates $G$.

\begin{lemma}\label{lem: number of component of Cay(G,A) equal number of component of Cay(G,genA)}
The numbers of components of $\Cay{G,A}$ and $\Cay{G,\gen{A}}$ are equal.
\end{lemma}

\begin{proof}
Set $E=\set{X:X~ \textrm{is a coset of}~ \gen{A}},~S=\set{C:C~ \textrm{is a component of}~ \Cay{G,A}}$, and $T=\set{C:C~ \textrm{is a component of}~ \Cay{G,\gen{A}}}$. By Corollary \ref{cor: component of Cay is in the form of coset}, $\abs{S}=\abs{E}=\abs{T}$.
\end{proof}

\begin{theorem}\label{thm: number of component Cay, subset}
The number of components of $\Cay{G,A}$ equals $[G:\gen{A}]$, the index of $\gen{A}$ in $G$.
\end{theorem}

\begin{proof}
The theorem follows directly from Corollary \ref{cor: component of Cay is in the form of coset} and Lemma \ref{lem: number of component of Cay(G,A) equal number of component of Cay(G,genA)}.
\end{proof}

As a consequence of Theorem \ref{thm: number of component Cay, subset}, we immediately obtain a few properties of \mbox{Cayley} graphs related to algebraic properties of groups.

\begin{corollary}
If $G$ is a finite group, then the number of components of $\Cay{G,A}$ divides $\abs{G}$.
\end{corollary}

\begin{proof}
This follows from Theorem \ref{thm: number of component Cay, subset} and the fact that $[G:\gen{A}]$ divides $\abs{G}$.
\end{proof}

\begin{corollary}\label{cor: characterization of connectivity in Cayly graph}
The Cayley graph $\Cay{G,A}$ is connected if and only if $G=\gen{A}$.
\end{corollary}

\begin{proof}
This follows from the fact that $[G:\gen{A}]=1$ if and only if $\gen{A}=G$.
\end{proof}

\begin{corollary}
Let $G$ be a group. Then $G$ is cyclic if and only if there exists an element $a\in G$ such that $\Cay{G,\set{a}}$ is connected.
\end{corollary}

Among other things, we obtain  a graph-theoretic version of the famous \mbox{Lagrange} theorem in abstract algebra, as shown in the following theorem.

\begin{theorem}\label{thm: Main Theorem}
Let $G$ be a group and let $H$ be a subgroup of $G$. Then the following hold:
\begin{enumerate}
\item\label{item: vertex component equal left coset} Each component of $\Cay{G, H}$ has a left coset of $H$ as its vertex set and is the complete graph $K_{\abs{H}}$. In particular, there is a one-to-one correspondence between the vertex sets of components of $\Cay{G, H}$ and the left cosets of $H$ in $G$.
\item\label{item: number of components} The Cayley graph $\Cay{G, H}$ has $[G\colon H]$ components. Hence, if $H$ is proper in $G$, then $\Cay{G, H}$ is disconnected.
\end{enumerate}
\end{theorem}

In view of Theorem \ref{thm: Main Theorem}, the index formula $\abs{G} = [G\colon H]\abs{H}$ can be recovered by counting the number of vertices of $\Cay{G, H}$ in the case when $G$ is finite. Moreover, a simple application of Theorem \ref{thm: Main Theorem} shows that $2$ always divides $\abs{G}(\abs{H}-1)$ for any subgroup $H$ of a finite group $G$. In fact, the result is trivial when $H = \set{e}$. Therefore, we assume that $H\ne\set{e}$. Let $C$ be a component of $\Cay{G, H}$. Then $C$ is the complete graph $K_{\abs{H}}$ and so there are ${\abs{H}\choose 2} = \frac{\abs{H}(\abs{H}-1)}{2}$ edges in $C$. Hence, the total number of edges in $\Cay{G, H}$ equals $$\frac{\abs{H}(\abs{H}-1)}{2}[G\colon H] = \frac{\abs{G}(\abs{H}-1)}{2}.$$
This shows that $\frac{\abs{G}(\abs{H}-1)}{2}$ must be an integer.

The next theorem shows how to construct a generating set of $G$ from an arbitrary subset $A$ of $G$ whenever $\Cay{G, A}$ has a finite number of components (e.g., $G$ is finite or $G$ has a subgroup of finite index).

\begin{theorem}\label{thm: generating set of digraph k component}
If $\Cay{G,A}$ has finitely many components $C_1, C_2,\ldots, C_k$ and if $v_i$ is a vertex in $C_i$ for all $i = 1,2,\ldots, k$, then
\begin{equation*}
S_1 = A\cup\set{v_1^{-1}v_2, v_2^{-1}v_3,\ldots, v_{k-1}^{-1}v_k}\quad\textrm{and}\quad S_2 = A\cup\set{v_1^{-1}v_2, v_1^{-1}v_3,\ldots, v_1^{-1}v_k}
\end{equation*}
form generating sets of $G$.
\end{theorem}

\begin{proof}
First, we will show that $\Cay{G,S_1}$ is connected. Let $u$ and $v$ be distinct vertices in $\Cay{G,S_1}$. If $u$ and $v$ are in the same component of $\Cay{G,A}$, then there is a path from $u$ to $v$ in $\Cay{G,A}$. Since $\Cay{G,A}$ is a subgraph of $\Cay{G,S_1}$, there is a path from $u$ to $v$ in $\Cay{G,S_1}$. Therefore, we may suppose that $u$ and $v$ are in distinct components of $\Cay{G,A}$, namely the $i^{\rm th}$ and $j^{\rm th}$ components, respectively. Hence, $u\p_1 v_i$ and $v\p_1 v_j$, where $\p_1$ is the equivalence relation induced by $\Cay{G,A}$. It follows that $u\p_2 v_i$ and $v\p_2 v_j$, where $\p_2$ is the equivalence relation induced by $\Cay{G,S_1}$. Since $v_{i+1}=v_i(v_i^{-1}v_{i+1})$ and $v_i^{-1}v_{i+1}\ne e$ for all $i=1,2,\ldots,k-1$, we obtain that $\set{v_i,v_{i+1}}$ is an edge in $\Cay{G,S_1}$ for all $i=1,2,\ldots,k-1$. This implies that $v_i\p_2 v_j$. By symmetry and transitivity, $u\p_2 v$ and so there is a path from $u$ to $v$ in $\Cay{G,S_1}$. Thus $\Cay{G,S_1}$ is connected. The verification that $S_2$ is a generating set of $G$ is similar to the case of $S_1$.
\end{proof}

\begin{theorem}
Let $A$ be a finite subset of a group $G$ not containing $e$. If $\Cay{G,A}$ has $k$ components, where $k\in\N$, then $G$ is generated by $\abs{A}+k-1$ elements and so
$$
\mathrm{rank}(G) \leq \abs{A}+k-1.
$$
\end{theorem}

\begin{proof}
Let $S_2$ be the set defined in Theorem \ref{thm: generating set of digraph k component}. We claim that $\abs{S_2}=\abs{A}+k-1$. Let $B=\set{v_1^{-1}v_2, v_1^{-1}v_3,\ldots, v_1^{-1}v_k}$. By the left cancellation law, the elements in $B$ are all distinct. So $\abs{B}=k-1$. Next, we show that $A\cap B=\emptyset$. If there is an element $a\in A\cap B$, then $a=v_1^{-1}v_i$ for some $i\in\set{2,3,\ldots,k}$. Since $v_i=v_1(v_1^{-1}v_i)=v_1a$ and $a\in A\setminus\set{e}$, there is an edge from $v_1$ to $v_i$, a contradiction. Thus $A\cap B=\emptyset$. Hence, $$\abs{S_2}=\abs{A\cup B}=\abs{A}+\abs{B}-\abs{A\cap B}=\abs{A}+k-1.$$ By Theorem \ref{thm: generating set of digraph k component}, $G$ is generated by $S_2$. By definition,
$$
\mathrm{rank}(G) = \min{\cset{\abs{X}}{X\subseteq G\textrm{ and }\gen{X} = G}} \leq \abs{A}+k-1,
$$
which completes the proof. 
\end{proof}

\begin{corollary}
Let $G$ be a group and let $a\in G\setminus\set{e}$. If $\Cay{G,\set{a}}$ has $k$ components, where $k\in\N$, then $G$ is generated by $k$ elements and so $\mathrm{rank}(G) \leq k$.
\end{corollary}

\section{An application to finding minimal generating sets}
A generating set $A$ of a (finite or infinite) group $G$ is {\it minimal} if no proper subset of $A$ generates $G$. It is clear that any finitely generated group has a minimal generating set, but finding one is quite difficult in certain circumstances. In this section, we provide an algorithm for finding minimal generating sets of finite groups as an application of Theorem \ref{thm: generating set of digraph k component}. Let $G$ be a finite group. A formal presentation of this algorithm is as follows.

\begin{enumerate}
    \item Set $A:=\set{a_1}$, where $a_1\in G$ and $a_1\neq e$.
    \item Set $v_1:=a_1,i:=1$.
    \item\label{item: skip process} Draw $\Cay{G,A}$.
    \item If $\Cay{G,A}$ is connected, skip to step \eqref{item: stop process}. Otherwise, set $i:=i+1$ and $v_2:=b_i$, where $b_i$ is an element of $G$ not in the component of $v_1$.
    \item Set $a_i:=v_1^{-1}v_2$ and $A:=A\cup\set{a_i}$.
    \item Return to step \eqref{item: skip process}.
    \item\label{item: stop process} If $i=1$, stop. Otherwise, set $i:=i-1$.
    \item Draw $\Cay{G,A\backslash\set{a_i}}$.
    \item If $\Cay{G,A\backslash\set{a_i}}$ is connected, set $A:=A\backslash\set{a_i}$. Otherwise, go to step \eqref{item: stop}.
    \item\label{item: stop} Return to step \eqref{item: stop process}.
\end{enumerate}

Theorem \ref{thm: generating set of digraph k component} ensures that this algorithm must stop at some point and turns $A$ into a minimal generating set of $G$. We illustrate how this algorithm works in the next example.

\begin{example}
Let $G$ be the group defined by presentation
\begin{equation}\label{presentation of group G}
G = \gen{a, b, c\colon a^2 = b^2 = (ab)^2= c^3 = acabc^{-1} =abcbc^{-1}}.
\end{equation}
Its Cayley table is given by Table \ref{tab: group K4:C3(A4)} (cf. \cite{GAP4.10.0}).

\begin{table}[ht]\centering
			\begin{tabular}{|c|c|c|c|c|c|c|c|c|c|c|c|c|}\hline
				$\cdot$ & $e$ & $a$ & $b$ & $ab$ & $c$ & $ac$ & $bc$ & $abc$ & $cc$ & $acc$ & $bcc$ & $abcc$\\ \hline
				$e$ & $e$ & $a$ & $b$ & $ab$ & $c$ & $ac$ & $bc$ & $abc$ & $cc$ & $acc$ & $bcc$ & $abcc$\\ \hline
				$a$ & $a$ & $e$ & $ab$ & $b$ & $ac$ & $c$ & $abc$ & $bc$ & $acc$ & $cc$ & $abcc$ & $bcc$\\ \hline
				$b$ & $b$ & $ab$ & $e$ & $a$ & $bc$ & $abc$ & $c$ & $ac$ & $bcc$ & $abcc$ & $cc$ & $acc$\\ \hline
				$ab$ & $ab$ & $b$ & $a$ & $e$ & $abc$ & $bc$ & $ac$ & $c$ & $abcc$ & $bcc$ & $acc$ & $cc$\\ \hline
				$c$ & $c$ & $bc$ & $abc$ & $ac$ & $cc$ & $bcc$ & $abcc$ & $acc$ & $e$ & $b$ & $ab$ & $a$\\ \hline
				$ac$ & $ac$ & $abc$ & $bc$ & $c$ & $acc$ & $abcc$ & $bcc$ & $cc$ & $a$ & $ab$ & $b$ & $e$\\ \hline
				$bc$ & $bc$ & $c$ & $ac$ & $abc$ & $bcc$ & $cc$ & $acc$ & $abcc$ & $b$ & $e$ & $a$ & $ab$\\ \hline
				$abc$ & $abc$ & $ac$ & $c$ & $bc$ & $abcc$ & $acc$ & $cc$ & $bcc$ & $ab$ & $a$ & $e$ & $b$\\ \hline
				$cc$ & $cc$ & $abcc$ & $acc$ & $bcc$ & $e$ & $ab$ & $a$ & $b$ & $c$ & $abc$ & $ac$ & $bc$\\ \hline
				$acc$ & $acc$ & $bcc$ & $cc$ & $abcc$ & $a$ & $b$ & $e$ & $ab$ & $ac$ & $bc$ & $c$ & $abc$\\ \hline
				$bcc$ & $bcc$ & $acc$ & $abcc$ & $cc$ & $b$ & $a$ & $ab$ & $e$ & $bc$ & $ac$ & $abc$ & $c$\\ \hline
				$abcc$ & $abcc$ & $cc$ & $bcc$ & $acc$ & $ab$ & $e$ & $b$ & $a$ & $abc$ & $c$ & $bc$ & $ac$\\ \hline
			\end{tabular}
			\caption{Cayley table of the group $G$ defined by (\ref{presentation of group G}) (cf. \cite{GAP4.10.0}).}\label{tab: group K4:C3(A4)}
		\end{table}

We can use the algorithm mentioned previously to find a minimal generating set of $G$ as follows:
\begin{enumerate}
    \item Set $A:=\set{b}$.
    \item Set $v_1:=b,i:=1$.
    \item Draw $\Cay{G,\set{b}}$, as shown in Figure \ref{fig: dcCay{A4,{b}}}.
    \item Since $\Cay{G,\set{b}}$ is not connected, set $i:=2$ and $v_2:=ab$.
    \item Set $a_2:=b^{-1}(ab)=a$ and $A:=\set{b,a}$.
    \item Draw $\Cay{G,\set{b,a}}$, as shown in Figure \ref{fig: dcCay{A4,{b,a}}}.
    \item Since $\Cay{G,\set{b,a}}$ is not connected, set $i:=3$ and $v_2:=bc$.
    \item Set $a_3:=b^{-1}(bc)=c$ and $A:=\set{b,a,c}$.
    \item Draw $\Cay{G,\set{b,a,c}}$, as shown in Figure \ref{fig: dcCay{A4,{b,a,c}}}.
    \item Since $\Cay{G,\set{b,a,c}}$ is connected and $i=3$, set $i:=2$.
    \item Draw $\Cay{G,\set{b,c}}$, as shown in Figure \ref{fig: dcCay{A4,{b,c}}}.
    \item Since $\Cay{G,\set{b,c}}$ is connected, set $A:=\set{b,c}$.
    \item Since $i=2$, set $i:=1$.
    \item Draw $\Cay{G,\set{c}}$, as shown in Figure \ref{fig: dcCay{A4,{c}}}. 
    \item Since $\Cay{G,\set{c}}$ is not connected, go to the next step.
    \item Since $i=1$, stop.
    
\end{enumerate}
This shows that $A=\set{b,c}$ is a minimal generating set of $G$.
\end{example}
    
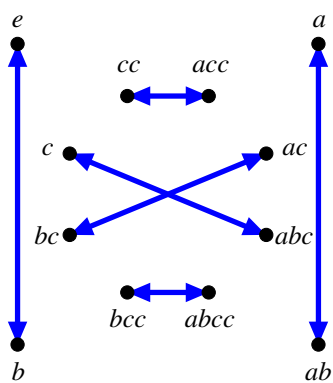
\begin{figure}[h]
\centering
\definecolor{qqqqff}{rgb}{0.,0.,1.}
\begin{tikzpicture}[line cap=round,line join=round,>=triangle 45,x=1.0cm,y=1.0cm]
\clip(-3.,-3.) rectangle (3.,3.);
\draw (-2.2,2.5) node[anchor=north west] {$e$};
\draw (1.8,2.5) node[anchor=north west] {$a$};
\draw (1.7,-2.1) node[anchor=north west] {$ab$};
\draw (-2.2,-2.1) node[anchor=north west] {$b$};
\draw (-0.8,1.9) node[anchor=north west] {$cc$};
\draw (0.2,1.9) node[anchor=north west] {$acc$};
\draw (1.4,0.8) node[anchor=north west] {$ac$};
\draw (1.3,-0.3) node[anchor=north west] {$abc$};
\draw (0.1,-1.4) node[anchor=north west] {$abcc$};
\draw (-0.9,-1.4) node[anchor=north west] {$bcc$};
\draw (-1.9,-0.3) node[anchor=north west] {$bc$};
\draw (-1.8,0.8) node[anchor=north west] {$c$};
\draw [line width=2.pt,color=qqqqff] (-2.,2.)-- (-2.,-2.);
\draw [line width=2.pt,color=qqqqff] (2.,2.)-- (2.,-2.);
\draw [line width=2.pt,color=qqqqff] (-0.5411961001461969,1.3065629648763766)-- (0.5411961001461971,1.3065629648763766);
\draw [line width=2.pt,color=qqqqff] (-0.5411961001461971,-1.3065629648763766)-- (0.5411961001461969,-1.3065629648763766);
\draw [line width=2.pt,color=qqqqff] (-1.3065629648763766,-0.541196100146197)-- (1.3065629648763766,0.5411961001461969);
\draw [line width=2.pt,color=qqqqff] (-1.3065629648763766,0.5411961001461969)-- (1.3065629648763766,-0.5411961001461971);
\draw [->,line width=1.pt,color=qqqqff] (-2.,-2.) -- (-2.,2.);
\draw [->,line width=1.pt,color=qqqqff] (-2.,2.) -- (-2.,-2.);
\draw [->,line width=1.pt,color=qqqqff] (2.,-2.) -- (2.,2.);
\draw [->,line width=1.pt,color=qqqqff] (2.,2.) -- (2.,-2.);
\draw [->,line width=1.pt,color=qqqqff] (-0.5411961001461969,1.3065629648763766) -- (0.5411961001461971,1.3065629648763766);
\draw [->,line width=1.pt,color=qqqqff] (0.5411961001461971,1.3065629648763766) -- (-0.5411961001461969,1.3065629648763766);
\draw [->,line width=1.pt,color=qqqqff] (-0.5411961001461971,-1.3065629648763766) -- (0.5411961001461969,-1.3065629648763766);
\draw [->,line width=1.pt,color=qqqqff] (0.5411961001461969,-1.3065629648763766) -- (-0.5411961001461971,-1.3065629648763766);
\draw [->,line width=1.pt,color=qqqqff] (-1.3065629648763766,-0.541196100146197) -- (1.3065629648763766,0.5411961001461968);
\draw [->,line width=1.pt,color=qqqqff] (1.3065629648763766,0.5411961001461969) -- (-1.3065629648763766,-0.5411961001461969);
\draw [->,line width=1.pt,color=qqqqff] (-1.3065629648763766,0.5411961001461969) -- (1.3065629648763766,-0.5411961001461971);
\draw [->,line width=1.pt,color=qqqqff] (1.3065629648763766,-0.5411961001461971) -- (-1.3065629648763766,0.5411961001461969);
\begin{scriptsize}
\draw [fill=black] (-2.,2.) circle (2.5pt);
\draw [fill=black] (2.,2.) circle (2.5pt);
\draw [fill=black] (2.,-2.) circle (2.5pt);
\draw [fill=black] (-2.,-2.) circle (2.5pt);
\draw [fill=black] (-1.3065629648763766,0.5411961001461969) circle (2.5pt);
\draw [fill=black] (-1.3065629648763766,-0.541196100146197) circle (2.5pt);
\draw [fill=black] (-0.5411961001461971,-1.3065629648763766) circle (2.5pt);
\draw [fill=black] (0.5411961001461969,-1.3065629648763766) circle (2.5pt);
\draw [fill=black] (1.3065629648763766,-0.5411961001461971) circle (2.5pt);
\draw [fill=black] (1.3065629648763766,0.5411961001461969) circle (2.5pt);
\draw [fill=black] (0.5411961001461971,1.3065629648763766) circle (2.5pt);
\draw [fill=black] (-0.5411961001461969,1.3065629648763766) circle (2.5pt);
\end{scriptsize}
\end{tikzpicture}
\caption{$\protect\overrightarrow{\mathrm{Cay}}_c(G,\set{b})$; blue arcs are induced by $b$.}
\label{fig: dcCay{A4,{b}}}
\end{figure}
    
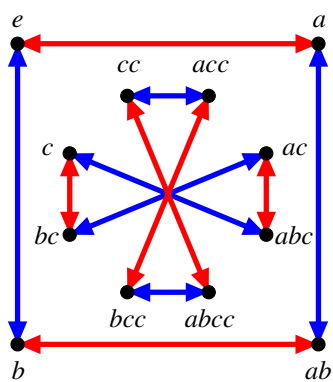
\begin{figure}[h]
\centering
\definecolor{ffqqqq}{rgb}{1.,0.,0.}
\definecolor{qqqqff}{rgb}{0.,0.,1.}
\begin{tikzpicture}[line cap=round,line join=round,>=triangle 45,x=1.0cm,y=1.0cm]
\clip(-3.,-3.) rectangle (3.,3.);
\draw (-2.2,2.5) node[anchor=north west] {$e$};
\draw (1.8,2.5) node[anchor=north west] {$a$};
\draw (1.7,-2.1) node[anchor=north west] {$ab$};
\draw (-2.2,-2.1) node[anchor=north west] {$b$};
\draw (-0.8,1.9) node[anchor=north west] {$cc$};
\draw (0.2,1.9) node[anchor=north west] {$acc$};
\draw (1.4,0.8) node[anchor=north west] {$ac$};
\draw (1.3,-0.3) node[anchor=north west] {$abc$};
\draw (0.1,-1.4) node[anchor=north west] {$abcc$};
\draw (-0.9,-1.4) node[anchor=north west] {$bcc$};
\draw (-1.9,-0.3) node[anchor=north west] {$bc$};
\draw (-1.8,0.8) node[anchor=north west] {$c$};
\draw [line width=2.pt,color=qqqqff] (-2.,2.)-- (-2.,-2.);
\draw [line width=2.pt,color=qqqqff] (2.,2.)-- (2.,-2.);
\draw [line width=2.pt,color=qqqqff] (-0.5411961001461969,1.3065629648763766)-- (0.5411961001461971,1.3065629648763766);
\draw [line width=2.pt,color=qqqqff] (-0.5411961001461971,-1.3065629648763766)-- (0.5411961001461969,-1.3065629648763766);
\draw [line width=2.pt,color=qqqqff] (-1.3065629648763766,-0.541196100146197)-- (1.3065629648763766,0.5411961001461969);
\draw [line width=2.pt,color=qqqqff] (-1.3065629648763766,0.5411961001461969)-- (1.3065629648763766,-0.5411961001461971);
\draw [->,line width=1.pt,color=qqqqff] (-2.,-2.) -- (-2.,2.);
\draw [->,line width=1.pt,color=qqqqff] (-2.,2.) -- (-2.,-2.);
\draw [->,line width=1.pt,color=qqqqff] (2.,-2.) -- (2.,2.);
\draw [->,line width=1.pt,color=qqqqff] (2.,2.) -- (2.,-2.);
\draw [->,line width=1.pt,color=qqqqff] (-0.5411961001461969,1.3065629648763766) -- (0.5411961001461971,1.3065629648763766);
\draw [->,line width=1.pt,color=qqqqff] (0.5411961001461971,1.3065629648763766) -- (-0.5411961001461969,1.3065629648763766);
\draw [->,line width=1.pt,color=qqqqff] (-0.5411961001461971,-1.3065629648763766) -- (0.5411961001461969,-1.3065629648763766);
\draw [->,line width=1.pt,color=qqqqff] (0.5411961001461969,-1.3065629648763766) -- (-0.5411961001461971,-1.3065629648763766);
\draw [->,line width=1.pt,color=qqqqff] (-1.3065629648763766,-0.541196100146197) -- (1.3065629648763766,0.5411961001461968);
\draw [->,line width=1.pt,color=qqqqff] (1.3065629648763766,0.5411961001461969) -- (-1.3065629648763766,-0.5411961001461969);
\draw [->,line width=1.pt,color=qqqqff] (-1.3065629648763766,0.5411961001461969) -- (1.3065629648763766,-0.5411961001461971);
\draw [->,line width=1.pt,color=qqqqff] (1.3065629648763766,-0.5411961001461971) -- (-1.3065629648763766,0.5411961001461969);
\draw [line width=2.pt,color=ffqqqq] (-2.,2.)-- (2.,2.);
\draw [line width=2.pt,color=ffqqqq] (-2.,-2.)-- (2.,-2.);
\draw [line width=2.pt,color=ffqqqq] (-1.3065629648763766,0.5411961001461969)-- (-1.3065629648763766,-0.541196100146197);
\draw [line width=2.pt,color=ffqqqq] (1.3065629648763766,0.5411961001461969)-- (1.3065629648763766,-0.5411961001461971);
\draw [line width=2.pt,color=ffqqqq] (-0.5411961001461971,-1.3065629648763766)-- (0.5411961001461971,1.3065629648763766);
\draw [line width=2.pt,color=ffqqqq] (-0.5411961001461969,1.3065629648763766)-- (0.5411961001461969,-1.3065629648763766);
\draw [->,line width=1.pt,color=ffqqqq] (-2.,2.) -- (2.,2.);
\draw [->,line width=1.pt,color=ffqqqq] (2.,2.) -- (-2.,2.);
\draw [->,line width=1.pt,color=ffqqqq] (-2.,-2.) -- (2.,-2.);
\draw [->,line width=1.pt,color=ffqqqq] (2.,-2.) -- (-2.,-2.);
\draw [->,line width=1.pt,color=ffqqqq] (-1.3065629648763766,-0.541196100146197) -- (-1.3065629648763766,0.5411961001461968);
\draw [->,line width=1.pt,color=ffqqqq] (-1.3065629648763766,0.5411961001461969) -- (-1.3065629648763766,-0.5411961001461969);
\draw [->,line width=1.pt,color=ffqqqq] (1.3065629648763766,-0.5411961001461971) -- (1.3065629648763766,0.5411961001461969);
\draw [->,line width=1.pt,color=ffqqqq] (1.3065629648763766,0.5411961001461969) -- (1.3065629648763766,-0.5411961001461971);
\draw [->,line width=1.pt,color=ffqqqq] (-0.5411961001461971,-1.3065629648763766) -- (0.5411961001461971,1.3065629648763766);
\draw [->,line width=1.pt,color=ffqqqq] (0.5411961001461971,1.3065629648763766) -- (-0.5411961001461971,-1.3065629648763766);
\draw [->,line width=1.pt,color=ffqqqq] (-0.5411961001461969,1.3065629648763766) -- (0.5411961001461969,-1.3065629648763766);
\draw [->,line width=1.pt,color=ffqqqq] (0.5411961001461969,-1.3065629648763766) -- (-0.5411961001461969,1.3065629648763766);
\begin{scriptsize}
\draw [fill=black] (-2.,2.) circle (2.5pt);
\draw [fill=black] (2.,2.) circle (2.5pt);
\draw [fill=black] (2.,-2.) circle (2.5pt);
\draw [fill=black] (-2.,-2.) circle (2.5pt);
\draw [fill=black] (-1.3065629648763766,0.5411961001461969) circle (2.5pt);
\draw [fill=black] (-1.3065629648763766,-0.541196100146197) circle (2.5pt);
\draw [fill=black] (-0.5411961001461971,-1.3065629648763766) circle (2.5pt);
\draw [fill=black] (0.5411961001461969,-1.3065629648763766) circle (2.5pt);
\draw [fill=black] (1.3065629648763766,-0.5411961001461971) circle (2.5pt);
\draw [fill=black] (1.3065629648763766,0.5411961001461969) circle (2.5pt);
\draw [fill=black] (0.5411961001461971,1.3065629648763766) circle (2.5pt);
\draw [fill=black] (-0.5411961001461969,1.3065629648763766) circle (2.5pt);
\end{scriptsize}
\end{tikzpicture}
\caption{$\protect\overrightarrow{\mathrm{Cay}}_c(G,\set{b,a})$; blue arcs are induced by $b$ and red arcs are induced by $a$.} 
\label{fig: dcCay{A4,{b,a}}}
\end{figure}
    
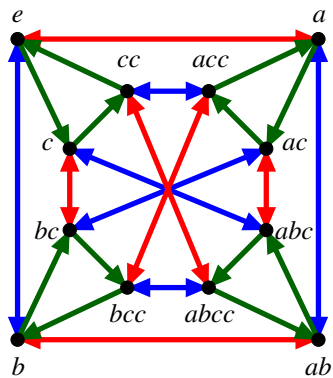
\begin{figure}[h]
\centering
\definecolor{qqwuqq}{rgb}{0.,0.39215686274509803,0.}
\definecolor{zzttqq}{rgb}{0.6,0.2,0.}
\definecolor{ffqqqq}{rgb}{1.,0.,0.}
\definecolor{qqqqff}{rgb}{0.,0.,1.}
\begin{tikzpicture}[line cap=round,line join=round,>=triangle 45,x=1.0cm,y=1.0cm]
\clip(-3.,-3.) rectangle (3.,3.);
\draw (-2.2,2.5) node[anchor=north west] {$e$};
\draw (1.8,2.5) node[anchor=north west] {$a$};
\draw (1.7,-2.1) node[anchor=north west] {$ab$};
\draw (-2.2,-2.1) node[anchor=north west] {$b$};
\draw (-0.8,1.9) node[anchor=north west] {$cc$};
\draw (0.2,1.9) node[anchor=north west] {$acc$};
\draw (1.4,0.8) node[anchor=north west] {$ac$};
\draw (1.3,-0.3) node[anchor=north west] {$abc$};
\draw (0.1,-1.4) node[anchor=north west] {$abcc$};
\draw (-0.9,-1.4) node[anchor=north west] {$bcc$};
\draw (-1.9,-0.3) node[anchor=north west] {$bc$};
\draw (-1.8,0.8) node[anchor=north west] {$c$};
\draw [line width=2.pt,color=qqqqff] (-2.,2.)-- (-2.,-2.);
\draw [line width=2.pt,color=qqqqff] (2.,2.)-- (2.,-2.);
\draw [line width=2.pt,color=qqqqff] (-0.5411961001461969,1.3065629648763766)-- (0.5411961001461971,1.3065629648763766);
\draw [line width=2.pt,color=qqqqff] (-0.5411961001461971,-1.3065629648763766)-- (0.5411961001461969,-1.3065629648763766);
\draw [line width=2.pt,color=qqqqff] (-1.3065629648763766,-0.541196100146197)-- (1.3065629648763766,0.5411961001461969);
\draw [line width=2.pt,color=qqqqff] (-1.3065629648763766,0.5411961001461969)-- (1.3065629648763766,-0.5411961001461971);
\draw [->,line width=1.pt,color=qqqqff] (-2.,-2.) -- (-2.,2.);
\draw [->,line width=1.pt,color=qqqqff] (-2.,2.) -- (-2.,-2.);
\draw [->,line width=1.pt,color=qqqqff] (2.,-2.) -- (2.,2.);
\draw [->,line width=1.pt,color=qqqqff] (2.,2.) -- (2.,-2.);
\draw [->,line width=1.pt,color=qqqqff] (-0.5411961001461969,1.3065629648763766) -- (0.5411961001461971,1.3065629648763766);
\draw [->,line width=1.pt,color=qqqqff] (0.5411961001461971,1.3065629648763766) -- (-0.5411961001461969,1.3065629648763766);
\draw [->,line width=1.pt,color=qqqqff] (-0.5411961001461971,-1.3065629648763766) -- (0.5411961001461969,-1.3065629648763766);
\draw [->,line width=1.pt,color=qqqqff] (0.5411961001461969,-1.3065629648763766) -- (-0.5411961001461971,-1.3065629648763766);
\draw [->,line width=1.pt,color=qqqqff] (-1.3065629648763766,-0.541196100146197) -- (1.3065629648763766,0.5411961001461968);
\draw [->,line width=1.pt,color=qqqqff] (1.3065629648763766,0.5411961001461969) -- (-1.3065629648763766,-0.5411961001461969);
\draw [->,line width=1.pt,color=qqqqff] (-1.3065629648763766,0.5411961001461969) -- (1.3065629648763766,-0.5411961001461971);
\draw [->,line width=1.pt,color=qqqqff] (1.3065629648763766,-0.5411961001461971) -- (-1.3065629648763766,0.5411961001461969);
\draw [line width=2.pt,color=ffqqqq] (-2.,2.)-- (2.,2.);
\draw [line width=2.pt,color=ffqqqq] (-2.,-2.)-- (2.,-2.);
\draw [line width=2.pt,color=ffqqqq] (-1.3065629648763766,0.5411961001461969)-- (-1.3065629648763766,-0.541196100146197);
\draw [line width=2.pt,color=ffqqqq] (1.3065629648763766,0.5411961001461969)-- (1.3065629648763766,-0.5411961001461971);
\draw [line width=2.pt,color=ffqqqq] (-0.5411961001461971,-1.3065629648763766)-- (0.5411961001461971,1.3065629648763766);
\draw [line width=2.pt,color=ffqqqq] (-0.5411961001461969,1.3065629648763766)-- (0.5411961001461969,-1.3065629648763766);
\draw [->,line width=1.pt,color=ffqqqq] (-2.,2.) -- (2.,2.);
\draw [->,line width=1.pt,color=ffqqqq] (2.,2.) -- (-2.,2.);
\draw [->,line width=1.pt,color=ffqqqq] (-2.,-2.) -- (2.,-2.);
\draw [->,line width=1.pt,color=ffqqqq] (2.,-2.) -- (-2.,-2.);
\draw [->,line width=1.pt,color=ffqqqq] (-1.3065629648763766,-0.541196100146197) -- (-1.3065629648763766,0.5411961001461968);
\draw [->,line width=1.pt,color=ffqqqq] (-1.3065629648763766,0.5411961001461969) -- (-1.3065629648763766,-0.5411961001461969);
\draw [->,line width=1.pt,color=ffqqqq] (1.3065629648763766,-0.5411961001461971) -- (1.3065629648763766,0.5411961001461969);
\draw [->,line width=1.pt,color=ffqqqq] (1.3065629648763766,0.5411961001461969) -- (1.3065629648763766,-0.5411961001461971);
\draw [->,line width=1.pt,color=ffqqqq] (-0.5411961001461971,-1.3065629648763766) -- (0.5411961001461971,1.3065629648763766);
\draw [->,line width=1.pt,color=ffqqqq] (0.5411961001461971,1.3065629648763766) -- (-0.5411961001461971,-1.3065629648763766);
\draw [->,line width=1.pt,color=ffqqqq] (-0.5411961001461969,1.3065629648763766) -- (0.5411961001461969,-1.3065629648763766);
\draw [->,line width=1.pt,color=ffqqqq] (0.5411961001461969,-1.3065629648763766) -- (-0.5411961001461969,1.3065629648763766);
\draw [line width=2.pt,color=qqwuqq] (-2.,2.)-- (-0.5411961001461969,1.3065629648763766);
\draw [line width=2.pt,color=qqwuqq] (-0.5411961001461969,1.3065629648763766)-- (-1.3065629648763766,0.5411961001461969);
\draw [line width=2.pt,color=qqwuqq] (-1.3065629648763766,0.5411961001461969)-- (-2.,2.);
\draw [line width=2.pt,color=qqwuqq] (0.5411961001461971,1.3065629648763766)-- (2.,2.);
\draw [line width=2.pt,color=qqwuqq] (2.,2.)-- (1.3065629648763766,0.5411961001461969);
\draw [line width=2.pt,color=qqwuqq] (1.3065629648763766,0.5411961001461969)-- (0.5411961001461971,1.3065629648763766);
\draw [line width=2.pt,color=qqwuqq] (0.5411961001461969,-1.3065629648763766)-- (1.3065629648763766,-0.5411961001461971);
\draw [line width=2.pt,color=qqwuqq] (1.3065629648763766,-0.5411961001461971)-- (2.,-2.);
\draw [line width=2.pt,color=qqwuqq] (2.,-2.)-- (0.5411961001461969,-1.3065629648763766);
\draw [line width=2.pt,color=qqwuqq] (-2.,-2.)-- (-0.5411961001461971,-1.3065629648763766);
\draw [line width=2.pt,color=qqwuqq] (-0.5411961001461971,-1.3065629648763766)-- (-1.3065629648763766,-0.541196100146197);
\draw [line width=2.pt,color=qqwuqq] (-1.3065629648763766,-0.541196100146197)-- (-2.,-2.);
\draw [->,line width=1.pt,color=qqwuqq] (-2.,2.) -- (-1.3065629648763766,0.5411961001461969);
\draw [->,line width=1.pt,color=qqwuqq] (-1.3065629648763766,0.5411961001461969) -- (-0.5411961001461969,1.3065629648763766);
\draw [->,line width=1.pt,color=qqwuqq] (-0.5411961001461969,1.3065629648763766) -- (-2.,2.);
\draw [->,line width=1.pt,color=qqwuqq] (0.5411961001461971,1.3065629648763766) -- (2.,2.);
\draw [->,line width=1.pt,color=qqwuqq] (2.,2.) -- (1.3065629648763766,0.5411961001461969);
\draw [->,line width=1.pt,color=qqwuqq] (1.3065629648763766,0.5411961001461969) -- (0.5411961001461971,1.3065629648763766);
\draw [->,line width=1.pt,color=qqwuqq] (-2.,-2.) -- (-1.3065629648763766,-0.5411961001461969);
\draw [->,line width=1.pt,color=qqwuqq] (-1.3065629648763766,-0.541196100146197) -- (-0.5411961001461971,-1.3065629648763766);
\draw [->,line width=1.pt,color=qqwuqq] (-0.5411961001461971,-1.3065629648763766) -- (-2.,-2.);
\draw [->,line width=1.pt,color=qqwuqq] (1.3065629648763766,-0.5411961001461971) -- (0.5411961001461969,-1.3065629648763766);
\draw [->,line width=1.pt,color=qqwuqq] (0.5411961001461969,-1.3065629648763766) -- (2.,-2.);
\draw [->,line width=1.pt,color=qqwuqq] (2.,-2.) -- (1.3065629648763766,-0.5411961001461971);
\begin{scriptsize}
\draw [fill=black] (-2.,2.) circle (2.5pt);
\draw [fill=black] (2.,2.) circle (2.5pt);
\draw [fill=black] (2.,-2.) circle (2.5pt);
\draw [fill=black] (-2.,-2.) circle (2.5pt);
\draw [fill=black] (-1.3065629648763766,0.5411961001461969) circle (2.5pt);
\draw [fill=black] (-1.3065629648763766,-0.541196100146197) circle (2.5pt);
\draw [fill=black] (-0.5411961001461971,-1.3065629648763766) circle (2.5pt);
\draw [fill=black] (0.5411961001461969,-1.3065629648763766) circle (2.5pt);
\draw [fill=black] (1.3065629648763766,-0.5411961001461971) circle (2.5pt);
\draw [fill=black] (1.3065629648763766,0.5411961001461969) circle (2.5pt);
\draw [fill=black] (0.5411961001461971,1.3065629648763766) circle (2.5pt);
\draw [fill=black] (-0.5411961001461969,1.3065629648763766) circle (2.5pt);
\end{scriptsize}
\end{tikzpicture}
\caption{$\protect\overrightarrow{\mathrm{Cay}}_c(G,\set{b,a,c})$; blue arcs are induced by $b$, red arcs are induced by $a$, and green arcs are induced by $c$.}
\label{fig: dcCay{A4,{b,a,c}}}
\end{figure}
    
\begin{figure}[h]
\centering
\definecolor{qqwuqq}{rgb}{0.,0.39215686274509803,0.}
\definecolor{zzttqq}{rgb}{0.6,0.2,0.}
\definecolor{qqqqff}{rgb}{0.,0.,1.}
\begin{tikzpicture}[line cap=round,line join=round,>=triangle 45,x=1.0cm,y=1.0cm]
\clip(-3.,-3.) rectangle (3.,3.);
\draw (-2.2,2.5) node[anchor=north west] {$e$};
\draw (1.8,2.5) node[anchor=north west] {$a$};
\draw (1.7,-2.1) node[anchor=north west] {$ab$};
\draw (-2.2,-2.1) node[anchor=north west] {$b$};
\draw (-0.8,1.9) node[anchor=north west] {$cc$};
\draw (0.2,1.9) node[anchor=north west] {$acc$};
\draw (1.4,0.8) node[anchor=north west] {$ac$};
\draw (1.3,-0.3) node[anchor=north west] {$abc$};
\draw (0.1,-1.4) node[anchor=north west] {$abcc$};
\draw (-0.9,-1.4) node[anchor=north west] {$bcc$};
\draw (-1.9,-0.3) node[anchor=north west] {$bc$};
\draw (-1.8,0.8) node[anchor=north west] {$c$};
\draw [line width=2.pt,color=qqqqff] (-2.,2.)-- (-2.,-2.);
\draw [line width=2.pt,color=qqqqff] (2.,2.)-- (2.,-2.);
\draw [line width=2.pt,color=qqqqff] (-0.5411961001461969,1.3065629648763766)-- (0.5411961001461971,1.3065629648763766);
\draw [line width=2.pt,color=qqqqff] (-0.5411961001461971,-1.3065629648763766)-- (0.5411961001461969,-1.3065629648763766);
\draw [line width=2.pt,color=qqqqff] (-1.3065629648763766,-0.541196100146197)-- (1.3065629648763766,0.5411961001461969);
\draw [line width=2.pt,color=qqqqff] (-1.3065629648763766,0.5411961001461969)-- (1.3065629648763766,-0.5411961001461971);
\draw [->,line width=1.pt,color=qqqqff] (-2.,-2.) -- (-2.,2.);
\draw [->,line width=1.pt,color=qqqqff] (-2.,2.) -- (-2.,-2.);
\draw [->,line width=1.pt,color=qqqqff] (2.,-2.) -- (2.,2.);
\draw [->,line width=1.pt,color=qqqqff] (2.,2.) -- (2.,-2.);
\draw [->,line width=1.pt,color=qqqqff] (-0.5411961001461969,1.3065629648763766) -- (0.5411961001461971,1.3065629648763766);
\draw [->,line width=1.pt,color=qqqqff] (0.5411961001461971,1.3065629648763766) -- (-0.5411961001461969,1.3065629648763766);
\draw [->,line width=1.pt,color=qqqqff] (-0.5411961001461971,-1.3065629648763766) -- (0.5411961001461969,-1.3065629648763766);
\draw [->,line width=1.pt,color=qqqqff] (0.5411961001461969,-1.3065629648763766) -- (-0.5411961001461971,-1.3065629648763766);
\draw [->,line width=1.pt,color=qqqqff] (-1.3065629648763766,-0.541196100146197) -- (1.3065629648763766,0.5411961001461968);
\draw [->,line width=1.pt,color=qqqqff] (1.3065629648763766,0.5411961001461969) -- (-1.3065629648763766,-0.5411961001461969);
\draw [->,line width=1.pt,color=qqqqff] (-1.3065629648763766,0.5411961001461969) -- (1.3065629648763766,-0.5411961001461971);
\draw [->,line width=1.pt,color=qqqqff] (1.3065629648763766,-0.5411961001461971) -- (-1.3065629648763766,0.5411961001461969);
\draw [line width=2.pt,color=qqwuqq] (-2.,2.)-- (-0.5411961001461969,1.3065629648763766);
\draw [line width=2.pt,color=qqwuqq] (-0.5411961001461969,1.3065629648763766)-- (-1.3065629648763766,0.5411961001461969);
\draw [line width=2.pt,color=qqwuqq] (-1.3065629648763766,0.5411961001461969)-- (-2.,2.);
\draw [line width=2.pt,color=qqwuqq] (0.5411961001461971,1.3065629648763766)-- (2.,2.);
\draw [line width=2.pt,color=qqwuqq] (2.,2.)-- (1.3065629648763766,0.5411961001461969);
\draw [line width=2.pt,color=qqwuqq] (1.3065629648763766,0.5411961001461969)-- (0.5411961001461971,1.3065629648763766);
\draw [line width=2.pt,color=qqwuqq] (0.5411961001461969,-1.3065629648763766)-- (1.3065629648763766,-0.5411961001461971);
\draw [line width=2.pt,color=qqwuqq] (1.3065629648763766,-0.5411961001461971)-- (2.,-2.);
\draw [line width=2.pt,color=qqwuqq] (2.,-2.)-- (0.5411961001461969,-1.3065629648763766);
\draw [line width=2.pt,color=qqwuqq] (-2.,-2.)-- (-0.5411961001461971,-1.3065629648763766);
\draw [line width=2.pt,color=qqwuqq] (-0.5411961001461971,-1.3065629648763766)-- (-1.3065629648763766,-0.541196100146197);
\draw [line width=2.pt,color=qqwuqq] (-1.3065629648763766,-0.541196100146197)-- (-2.,-2.);
\draw [->,line width=1.pt,color=qqwuqq] (-2.,2.) -- (-1.3065629648763766,0.5411961001461969);
\draw [->,line width=1.pt,color=qqwuqq] (-1.3065629648763766,0.5411961001461969) -- (-0.5411961001461969,1.3065629648763766);
\draw [->,line width=1.pt,color=qqwuqq] (-0.5411961001461969,1.3065629648763766) -- (-2.,2.);
\draw [->,line width=1.pt,color=qqwuqq] (0.5411961001461971,1.3065629648763766) -- (2.,2.);
\draw [->,line width=1.pt,color=qqwuqq] (2.,2.) -- (1.3065629648763766,0.5411961001461969);
\draw [->,line width=1.pt,color=qqwuqq] (1.3065629648763766,0.5411961001461969) -- (0.5411961001461971,1.3065629648763766);
\draw [->,line width=1.pt,color=qqwuqq] (-2.,-2.) -- (-1.3065629648763766,-0.5411961001461969);
\draw [->,line width=1.pt,color=qqwuqq] (-1.3065629648763766,-0.541196100146197) -- (-0.5411961001461971,-1.3065629648763766);
\draw [->,line width=1.pt,color=qqwuqq] (-0.5411961001461971,-1.3065629648763766) -- (-2.,-2.);
\draw [->,line width=1.pt,color=qqwuqq] (1.3065629648763766,-0.5411961001461971) -- (0.5411961001461969,-1.3065629648763766);
\draw [->,line width=1.pt,color=qqwuqq] (0.5411961001461969,-1.3065629648763766) -- (2.,-2.);
\draw [->,line width=1.pt,color=qqwuqq] (2.,-2.) -- (1.3065629648763766,-0.5411961001461971);
\begin{scriptsize}
\draw [fill=black] (-2.,2.) circle (2.5pt);
\draw [fill=black] (2.,2.) circle (2.5pt);
\draw [fill=black] (2.,-2.) circle (2.5pt);
\draw [fill=black] (-2.,-2.) circle (2.5pt);
\draw [fill=black] (-1.3065629648763766,0.5411961001461969) circle (2.5pt);
\draw [fill=black] (-1.3065629648763766,-0.541196100146197) circle (2.5pt);
\draw [fill=black] (-0.5411961001461971,-1.3065629648763766) circle (2.5pt);
\draw [fill=black] (0.5411961001461969,-1.3065629648763766) circle (2.5pt);
\draw [fill=black] (1.3065629648763766,-0.5411961001461971) circle (2.5pt);
\draw [fill=black] (1.3065629648763766,0.5411961001461969) circle (2.5pt);
\draw [fill=black] (0.5411961001461971,1.3065629648763766) circle (2.5pt);
\draw [fill=black] (-0.5411961001461969,1.3065629648763766) circle (2.5pt);
\end{scriptsize}
\end{tikzpicture}
\caption{$\protect\overrightarrow{\mathrm{Cay}}_c(G,\set{b,c})$; blue arcs are induced by $b$ and green arcs are induced by $c$.}
\label{fig: dcCay{A4,{b,c}}}
\end{figure}
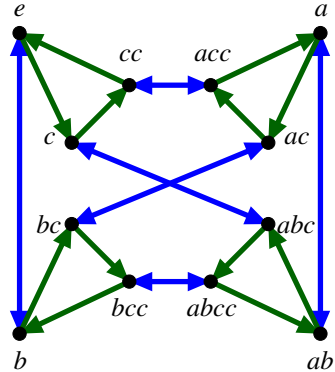
    
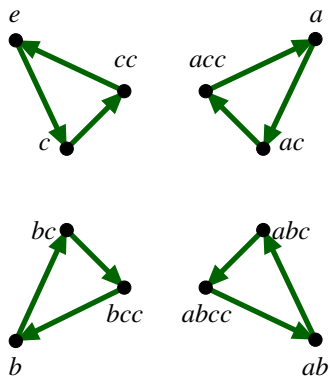
\begin{figure}[h]
\centering
\definecolor{qqwuqq}{rgb}{0.,0.39215686274509803,0.}
\definecolor{zzttqq}{rgb}{0.6,0.2,0.}
\begin{tikzpicture}[line cap=round,line join=round,>=triangle 45,x=1.0cm,y=1.0cm]
\clip(-3.,-3.) rectangle (3.,3.);
\draw (-2.2,2.5) node[anchor=north west] {$e$};
\draw (1.8,2.5) node[anchor=north west] {$a$};
\draw (1.7,-2.1) node[anchor=north west] {$ab$};
\draw (-2.2,-2.1) node[anchor=north west] {$b$};
\draw (-0.8,1.9) node[anchor=north west] {$cc$};
\draw (0.2,1.9) node[anchor=north west] {$acc$};
\draw (1.4,0.8) node[anchor=north west] {$ac$};
\draw (1.3,-0.3) node[anchor=north west] {$abc$};
\draw (0.1,-1.4) node[anchor=north west] {$abcc$};
\draw (-0.9,-1.4) node[anchor=north west] {$bcc$};
\draw (-1.9,-0.3) node[anchor=north west] {$bc$};
\draw (-1.8,0.8) node[anchor=north west] {$c$};
\draw [line width=2.pt,color=qqwuqq] (-1.9834292323733376,1.9834292323733376)-- (-0.5411961001461969,1.3065629648763766);
\draw [line width=2.pt,color=qqwuqq] (-0.5411961001461969,1.3065629648763766)-- (-1.3065629648763766,0.5411961001461969);
\draw [line width=2.pt,color=qqwuqq] (-1.3065629648763766,0.5411961001461969)-- (-1.9834292323733376,1.9834292323733376);
\draw [line width=2.pt,color=qqwuqq] (0.5411961001461971,1.3065629648763766)-- (2.,2.);
\draw [line width=2.pt,color=qqwuqq] (2.,2.)-- (1.3065629648763766,0.5411961001461969);
\draw [line width=2.pt,color=qqwuqq] (1.3065629648763766,0.5411961001461969)-- (0.5411961001461971,1.3065629648763766);
\draw [line width=2.pt,color=qqwuqq] (0.5411961001461969,-1.3065629648763766)-- (1.3065629648763766,-0.5411961001461971);
\draw [line width=2.pt,color=qqwuqq] (1.3065629648763766,-0.5411961001461971)-- (2.,-2.);
\draw [line width=2.pt,color=qqwuqq] (2.,-2.)-- (0.5411961001461969,-1.3065629648763766);
\draw [line width=2.pt,color=qqwuqq] (-1.9834292323733376,-2.0165707676266624)-- (-0.5411961001461971,-1.3065629648763766);
\draw [line width=2.pt,color=qqwuqq] (-0.5411961001461971,-1.3065629648763766)-- (-1.3065629648763766,-0.541196100146197);
\draw [line width=2.pt,color=qqwuqq] (-1.3065629648763766,-0.541196100146197)-- (-1.9834292323733376,-2.0165707676266624);
\draw [->,line width=1.pt,color=qqwuqq] (-1.9834292323733376,1.9834292323733376) -- (-1.3065629648763766,0.5411961001461969);
\draw [->,line width=1.pt,color=qqwuqq] (-1.3065629648763766,0.5411961001461969) -- (-0.5411961001461969,1.3065629648763766);
\draw [->,line width=1.pt,color=qqwuqq] (-0.5411961001461969,1.3065629648763766) -- (-1.9834292323733376,1.9834292323733376);
\draw [->,line width=1.pt,color=qqwuqq] (0.5411961001461971,1.3065629648763766) -- (2.,2.);
\draw [->,line width=1.pt,color=qqwuqq] (2.,2.) -- (1.3065629648763766,0.5411961001461969);
\draw [->,line width=1.pt,color=qqwuqq] (1.3065629648763766,0.5411961001461969) -- (0.5411961001461971,1.3065629648763766);
\draw [->,line width=1.pt,color=qqwuqq] (-1.9834292323733376,-2.0165707676266624) -- (-1.3065629648763766,-0.5411961001461969);
\draw [->,line width=1.pt,color=qqwuqq] (-1.3065629648763766,-0.541196100146197) -- (-0.5411961001461971,-1.3065629648763766);
\draw [->,line width=1.pt,color=qqwuqq] (-0.5411961001461971,-1.3065629648763766) -- (-1.9834292323733376,-2.0165707676266624);
\draw [->,line width=1.pt,color=qqwuqq] (1.3065629648763766,-0.5411961001461971) -- (0.5411961001461969,-1.3065629648763766);
\draw [->,line width=1.pt,color=qqwuqq] (0.5411961001461969,-1.3065629648763766) -- (2.,-2.);
\draw [->,line width=1.pt,color=qqwuqq] (2.,-2.) -- (1.3065629648763766,-0.5411961001461971);
\begin{scriptsize}
\draw [fill=black] (-1.9834292323733376,1.9834292323733376) circle (2.5pt);
\draw [fill=black] (2.,2.) circle (2.5pt);
\draw [fill=black] (2.,-2.) circle (2.5pt);
\draw [fill=black] (-1.9834292323733376,-2.0165707676266624) circle (2.5pt);
\draw [fill=black] (-1.3065629648763766,0.5411961001461969) circle (2.5pt);
\draw [fill=black] (-1.3065629648763766,-0.541196100146197) circle (2.5pt);
\draw [fill=black] (-0.5411961001461971,-1.3065629648763766) circle (2.5pt);
\draw [fill=black] (0.5411961001461969,-1.3065629648763766) circle (2.5pt);
\draw [fill=black] (1.3065629648763766,-0.5411961001461971) circle (2.5pt);
\draw [fill=black] (1.3065629648763766,0.5411961001461969) circle (2.5pt);
\draw [fill=black] (0.5411961001461971,1.3065629648763766) circle (2.5pt);
\draw [fill=black] (-0.5411961001461969,1.3065629648763766) circle (2.5pt);
\end{scriptsize}
\end{tikzpicture}
\caption{$\protect\overrightarrow{\mathrm{Cay}}_c(G,\set{c})$; green arcs are induced by $c$.}
\label{fig: dcCay{A4,{c}}}
\end{figure}

\par\noindent\textbf{Acknowledgments.}
The authors would like to express their special gratitude to Professor Milton Ferreira for his hospitality at the University of Aveiro, Portugal with financial support by Institute for Promotion of Teaching Science and Technology (IPST), Thailand, via Development and Promotion of Science and Technology Talents Project (DPST). This research was supported by the Research Center in Mathematics and Applied Mathematics, Chiang Mai University.

\bibliographystyle{amsplain}\addcontentsline{toc}{section}{References}
\bibliography{References}

\providecommand{\bysame}{\leavevmode\hbox to3em{\hrulefill}\thinspace}
\providecommand{\MR}{\relax\ifhmode\unskip\space\fi MR }
\providecommand{\MRhref}[2]{%
  \href{http://www.ams.org/mathscinet-getitem?mr=#1}{#2}
}
\providecommand{\href}[2]{#2}
\begin{thebibliography}{10}

\bibitem{MR2296106}
V.~Arvind and J.~Tor\'{a}n, \emph{The complexity of quasigroup isomorphism and
  the \mbox{minimum} generating set problem}, Algorithms and computation,
  Lecture Notes in Computer Science, vol. 4288, Springer, Berlin, 2006,
  pp.~233--242.

\bibitem{MR3814347}
P.~J. Cameron, A.~Lucchini, and C.~M. Roney-Dougal, \emph{Generating sets of
  finite groups}, Trans. Amer. Math. Soc. \textbf{370} (2018), no.~9,
  6751--6770.

\bibitem{MR1934292}
X.~G. Fang, C.~E. Praeger, and J.~Wang, \emph{On the automorphism groups of
  {C}ayley graphs of finite simple groups}, J. London Math. Soc. \textbf{66}
  (2002), no.~3, 563--578.

\bibitem{MR643291}
C.~D. Godsil, \emph{Connectivity of minimal {C}ayley graphs}, Arch. Math.
  (Basel) \textbf{37} (1981), no.~5, 473--476.

\bibitem{GAP4.10.0}
The~{\sf GAP} Group, \emph{{\sf GAP} -- {G}roups, {A}lgorithms, and
  {P}rogramming}, Version 4.10.0, 2018, http://www.gap-system.org.

\bibitem{MR2361465}
L.~Halbeisen, M.~Hamilton, and P.~R\r{o}\v{z}i\v{c}ka, \emph{Minimal generating
  sets of groups, rings, and fields}, Quaest. Math. \textbf{30} (2007), no.~3,
  355--363.

\bibitem{MR2080457}
A.~V. Kelarev, \emph{Labelled {C}ayley graphs and minimal automata}, Australas.
  J. \mbox{Combin}. \textbf{30} (2004), 95--101.

\bibitem{MR2228536}
\bysame, \emph{On {C}ayley graphs of inverse semigroups}, Semigroup Forum
  \textbf{72} (2006), no.~3, 411--418.

\bibitem{MR1957965}
A.~V. Kelarev and C.~E. Praeger, \emph{On transitive {C}ayley graphs of groups
  and semigroups}, European J. Combin. \textbf{24} (2003), 59--72.

\bibitem{MR1939667}
A.~V. Kelarev and S.~J. Quinn, \emph{A combinatorial property and {C}ayley
  graphs of semigroups}, Semigroup Forum \textbf{66} (2003), 89--96.

\bibitem{MR2548552}
A.~V. Kelarev, J.~Ryan, and J.~Yearwood, \emph{Cayley graphs as classifiers for
  data mining: {T}he influence of asymmetries}, Discrete Math. \textbf{309}
  (2009), no.~17, 5360--5369.

\bibitem{MR2574829}
B.~Khosravi and M.~Mahmoudi, \emph{On {C}ayley graphs of rectangular groups},
  \mbox{Discrete} Math. \textbf{310} (2010), no.~4, 804--811.

\bibitem{MR2320601}
W.~Klotz and T.~Sander, \emph{Some properties of unitary {C}ayley graphs},
  Electron. J. Combin. \textbf{14} (2007), Article R45, 12 pages.

\bibitem{MR2455531}
E.~Konstantinova, \emph{Some problems on {C}ayley graphs}, Linear Algebra Appl.
  \textbf{429} (2008), no.~11-12, 2754--2769.

\bibitem{MR572868}
J.~S. Leon, \emph{On an algorithm for finding a base and a strong generating
  set for a group given by generating permutations}, Math. Comp. \textbf{35}
  (1980), no.~151, 941--974.

\bibitem{MR1927074}
C.~H. Li, \emph{On isomorphisms of finite {C}ayley graphs---a survey}, Discrete
  Math. \textbf{256} (2002), no.~1-2, 301--334.

\bibitem{MR1955389}
Z.-P. Lu and M.-Y. Xu, \emph{On the normality of {C}ayley graphs of order
  {$pq$}}, Australas. J. Combin. \textbf{27} (2003), 81--93.

\bibitem{MR3073659}
A.~Lucchini, \emph{The largest size of a minimal generating set of a finite
  group}, Arch. Math. (Basel) \textbf{101} (2013), 1--8.

\bibitem{MR1289999}
A.~Lucchini and F.~Menegazzo, \emph{Computing a set of generators of minimal
  cardinality in a solvable group}, J. Symbolic Comput. \textbf{17} (1994),
  no.~5, 409--420.

\bibitem{MR2023255}
F.~Menegazzo, \emph{The number of generators of a finite group}, Proceedings of
  the {A}ll {I}reland {A}lgebra {D}ays, 2001 ({B}elfast), no.~50, 2003,
  pp.~117--128.

\bibitem{MR2548555}
S.~Panma, U.~Knauer, and Sr. Arworn, \emph{On transitive {C}ayley graphs of
  strong semilattices of right (left) groups}, Discrete Math. \textbf{309}
  (2009), no.~17, 5393--5403.

\bibitem{MR2885431}
J.~\v{S}iagiov\'{a} and J.~\v{S}ir\'{a}\v{n}, \emph{Approaching the {M}oore
  bound for diameter two by \mbox{{C}ayley} graphs}, J. Combin. Theory Ser. B
  \textbf{102} (2012), no.~2, 470--473.

\bibitem{MR2195356}
W.~Xiao and B.~Parhami, \emph{{C}ayley graphs as models of deterministic
  small-world networks}, Inform. Process. Lett. \textbf{97} (2006), no.~3,
  115--117.

\end{thebibliography}
\end{document}